\documentclass[12pt]{amsart}
\usepackage{amssymb}
\usepackage[mathscr]{eucal}
\usepackage{epsf}
\usepackage{epsfig}
\usepackage{color}
\DeclareGraphicsExtensions{.pstex,.eps,.epsi,.ps}

 \newtheorem{thm}{Theorem}[section]
 \newtheorem{cor}[thm]{Corollary}
 \newtheorem{lem}[thm]{Lemma}
 
 \theoremstyle{definition}
 
 \theoremstyle{remark}
 \newtheorem{rem}[thm]{Remark}
 \numberwithin{equation}{section}

 \newcommand{\Real}{\mathbb{R}}

 \newcommand{\tr}{\textbf{tr}}
 
 \newcommand{\vol}{\textbf{vol}}
 \newcommand{\ric}{\textbf{Rc}}

  \newcommand{\covar}[1]{\frac{\overline D}{d #1}}
 \newcommand{\Rm}{\textbf{Rm}}
 
 \newcommand{\bn}{\textbf{n}}
 
 \newcommand{\hess}{\nabla^2}
 
  \newcommand{\Sh}{\mathfrak S}
  \newcommand{\meanvec}{\vec{\mathfrak H}}

\begin{document}

\title[Generalized Li-Yau estimates]{Generalized Li-Yau estimates and Huisken's monotonicity formula}

\author{Paul W.Y. Lee}
\email{wylee@math.cuhk.edu.hk}
\address{Room 216, Lady Shaw Building, The Chinese University of Hong Kong, Shatin, Hong Kong}

\date{\today}

\maketitle

\begin{abstract}
We prove a generalization of the Li-Yau estimate for a broad class of second order linear parabolic equations. As a consequence,  we obtain a new Cheeger-Yau inequality and a new Harnack inequality for these equations. We also prove a Hamilton-Li-Yau estimate, which is a matrix version of the Li-Yau estimate, for these equations. This results in a generalization of Huisken's monotonicity formula for a family of evolving hypersurfaces. Finally, we also show that all these generalizations are sharp in the sense that the inequalities become equality for a family of fundamental solutions, which however different from the Gaussian heat kernels on which the equality was achieved in the classical case.
\end{abstract}


\section{Introduction}

The Harnack inequality is one of the most fundamental results in the regularity theory of non-linear elliptic and parabolic equations. In the case of linear parabolic equations in divergence form, this inequality was first done in \cite{Mo}. A sharp version of this inequality which takes into account the geometry of the underlying manifold was first done in \cite{LiYa}. In fact, the key result in \cite{LiYa} is a sharp gradient estimate, now known as the Li-Yau estimate, for linear parabolic equations on Riemannian manifolds with a lower bound on the Ricci curvature. The sharp Harnack inequality can be obtained by integrating this estimate along geodesics. Because of this, this estimate and its generalizations are called differential Harnack inequalities.

There are numerous generalizations of the Li-Yau estimate. In the case of geometric evolution equations, this includes the evolution equations for hypersurfaces \cite{Ha4,Ch1,An}, the Yamabe flow \cite{Ch2}, the Ricci flow \cite{Ha2,Br} and its K\"ahler analogue \cite{Ca,Ni2}. For a more detail account of these generalizations as well as further developments, see \cite{Ni1}.

In the case of the heat equation $\dot\rho_t=\Delta\rho_t$ on a Riemannian manifold of dimension $n$ with non-negative Ricci curvature, the Li-Yau estimate is the following inequality for any positive solution $\rho_t$
\begin{equation}\label{LiYau}
\Delta\log\rho_t\geq -\frac{n}{2t}
\end{equation}
for any time $t>0$.

This is sharp in the sense that the equality case of (\ref{LiYau}) is satisfied by the following solution of the heat equation on the Euclidean space:
\[
\rho_t(x)=\frac{1}{(4\pi t)^{n/2}}\exp\left(-\frac{|x|^2}{4t}\right),
\]
where $|x|$ is the Euclidean norm of $x$ in $\Real^n$.

On the other hand, there are also generalizations of the inequality (\ref{LiYau}) to other second order linear parabolic equations under the, so called, curvature-dimension conditions (see for instance \cite{BaLe}). They are estimates of the form
\begin{equation}\label{BLLiYau}
L\log\rho_t\geq -\frac{n}{2t},
\end{equation}
where $L$ is a linear differential operator without constant term and $\rho_t$ is a solution of the equation $\dot\rho_t=L\rho_t$.

However, the following
\begin{equation}\label{fund}
\rho_t(x)=\left(\frac{2\pi(\exp(2tk)-1)}{k\exp(2tk)}\right)^{-n/2}\exp\left(\frac{-k|x|^2}{2(\exp(2tk)-1)}\right).
\end{equation}
is a solution of the equation
\begin{equation}\label{sharp}
\dot\rho_t=\Delta \rho_t-k\left<x,\nabla\rho_t\right>=\Delta \rho_t-\left<\nabla \left(\frac{k}{2}|x|^2\right),\nabla\rho_t\right>
\end{equation}
where $k> 0$ is a constant.

The solutions (\ref{fund}) never satisfy the equality case of (\ref{BLLiYau}). Motivated by this observation, we prove the following generalization of (\ref{LiYau}).

\begin{thm}\label{main}(Generalized Li-Yau estimate)
Assume that the Ricci curvature of the underlying Riemannian manifold $M$ is non-negative. Let $U_1, U_2:M\to\Real$ be two smooth functions and let
\[
V:=\Delta U_1+\frac{1}{2}|\nabla U_1|^2-2U_2.
\]
Assume that $|\nabla V|$ is bounded and $\Delta V\leq nk^2$. Let $\rho_t$ be a positive solution of the equation
\[
\dot\rho_t=\Delta\rho_t+\left<\nabla U_1,\nabla\rho_t\right>+U_2\rho_t.
\]
Then $\rho_t$ satisfies
\begin{equation}\label{newLiYau}
\Delta\log\rho_t+\frac{1}{2}\Delta U_1\geq -\frac{nk}{2}\coth(kt).
\end{equation}
\end{thm}

By letting $U_1\equiv 0$, $U_2\equiv 0$ and $k$ goes to $0$, we recover the estimate (\ref{LiYau}). Note also that the solution (\ref{fund}) achieves the equality case of (\ref{newLiYau}) and the assumptions of Theorem \ref{main} with $U_1=-\frac{k}{2}|x|^2$ and $U_2\equiv 0$.

Recall that we can obtain the Harnack inequality by integrating (\ref{LiYau}) along geodesics. An analogue of this fact also holds true in our setting. However, instead of integrating along geodesics, the correct paths in this case are the minimizers of the following functional:
\begin{equation}\label{cost}
c_{s,t}(x,y)=\inf_{\gamma(s)=x,\gamma(t)=y}\int_s^t\frac{1}{2}|\dot\gamma(\tau)|^2+V(\gamma(\tau))d\tau,
\end{equation}
where the infimum is taken over all paths $\gamma:[s,t]\to M$ joining $x$ and $y$ and $V=\Delta U_1+\frac{1}{2}|\nabla U_1|^2-2U_2$. The idea of considering functionals of the form (\ref{cost}) already appeared in \cite{LiYa}. In the case of the Ricci flow, a version of the cost function (\ref{cost}), called $L$-distance, appeared in \cite{Pe}.

\begin{thm}\label{newHarnack}(Generalized Harnack inequality)
Under the same assumptions as in Theorem \ref{main} and that $V$ is bounded below, the following estimate holds:
\[
\frac{\rho_t(y)}{\rho_s(x)}\geq \left(\frac{\sinh(kt)}{\sinh(ks)}\right)^{-\frac{n}{2}}\exp\left(-\frac{1}{2}\left(c_{s,t}(x,y)+U_1(y)-U_1(x)\right)\right).
\]
\end{thm}

By letting $U_1\equiv 0$, $U_2\equiv 0$, and $k$ goes to $0$, we recover the following Harnack estimate.

\begin{cor}\label{Harnack}(The Harnack inequality \cite{Mo,LiYa})
Assume that the Ricci curvature of the underlying Riemannian manifold $M$ is non-negative. Then any positive solution $\rho_t$ of the equation $\dot\rho_t=\Delta\rho_t$ satisfies the following estimate:
\[
\frac{\rho_t(y)}{\rho_s(x)}\geq \left(\frac{t}{s}\right)^{-n/2}e^{-\frac{d^2(x,y)}{4(t-s)}}.
\]
\end{cor}

It is also known that Corollary \ref{Harnack} recovers the heat kernel comparison theorem of Cheeger-Yau \cite{ChYa} if we let $\rho_t$ be the heat kernel and letting $s$ goes to $0$. The same principle also works for Theorem \ref{newHarnack}.

\begin{thm}\label{newCheYau}(Generalized Cheeger-Yau comparison theorem)
Under the same assumptions as in Theorem \ref{newHarnack}, the following estimate holds for the fundamental solution $p_t$ of the equation  $\dot\rho_t=\Delta\rho_t+\left<\nabla U_1,\nabla\rho_t\right>+U_2\rho_t$:
\begin{equation}\label{newCYest}
\begin{split}
p_t(x,y)\geq &\left(\frac{k}{4\pi\sinh(kt)}\right)^{\frac{n}{2}}\\
&\cdot\exp\left(-\frac{1}{2}\left(c_{0,t}(x,y)+U_1(y)-U_1(x)\right)\right).
\end{split}
\end{equation}
\end{thm}

In the case $U_1=-\frac{k|x|^2}{2}$ and $U_2\equiv 0$, the cost function is given by
\[
c_{0,t}(0,y)=\frac{k|y|^2\coth(kt)}{2}-knt
\]
(see the proof of Theorem \ref{maincor4}) and right hand side of (\ref{newCYest}) becomes the fundamental solution (\ref{fund}). Therefore, all inequalities in Theorem \ref{newCheYau} become equalities in this case.

Again, by setting $U_1\equiv 0$, $U_2\equiv 0$ and letting $k$ goes to $0$, we recover the Cheeger-Yau estimate.

\begin{cor}\label{CheYau}\cite{ChYa}(The Cheeger-Yau heat kernel comparison)
Assume that the Ricci curvature of the underlying Riemannian manifold $M$ is non-negative. Then the heat kernel $p_t(x,y)$ of the equation $\dot\rho_t=\Delta\rho_t$ satisfy the following estimate:
\[
p_t(x,y)\geq \frac{1}{(4\pi t)^{n/2}}e^{-\frac{d^2(x,y)}{4t}}.
\]
\end{cor}

As another consequence of Theorem \ref{main}, we obtain the following Liouville type theorem.

\begin{cor}\label{Liouville}(A Liouville type theorem)
Assume that the Ricci curvature of the underlying Riemannian manifold $M$ is non-negative. Suppose that $|\nabla V|$ is bounded and $\Delta V\leq nk^2$.  Then any positive solution $\rho$ of the equation
\begin{equation}\label{elliptic}
\Delta\rho+\left<\nabla U_1,\nabla\rho\right>+U_2\rho=0
\end{equation}
satisfies
\[
\left|\nabla\log\rho+\frac{1}{2}\nabla U_1\right|^2\leq\frac{1}{2}V.
\]
In particular, if $V(x)<0$ at some point $x$ in $M$, then the equation (\ref{elliptic}) does not admit any positive solution. If $V\equiv 0$, then there is a positive constant $C$ such that
\[
\rho=Ce^{-\frac{1}{2}U_1}.
\]
\end{cor}

As a special case of Corollary \ref{Liouville}, we recover the following result in \cite{Ya}.

\begin{cor}\label{Liouville}(The Liouville theorem)
Assume that the Ricci curvature of the underlying Riemannian manifold $M$ is non-negative. Then any non-negative harmonic function is a constant.
\end{cor}

In \cite{Ha1}, Hamilton proved a matrix version of (\ref{LiYau}) for the heat equation, called the Hamilton-Li-Yau estimate (see also a K\"ahler analogue in \cite{Ni1}). Another matrix version of the differential Harnack inequality also appeared in \cite{Ha2} which is one of the most fundamental result in the theory of the Ricci flow (see also an interesting generalization in \cite{Br} and a K\"ahler analogue in \cite{CaNi}). The following is a matrix version of (\ref{main}).

\begin{thm}\label{newLiYauHam}(Generalized Hamilton-Li-Yau estimate)
Assume that the sectional curvature of the underlying compact Riemannian manifold $M$ is non-negative and the Ricci curvature is parallel. Let $U_1, U_2:M\to\Real$ be two smooth functions satisfying the following condition for some non-negative constant $k$:
\[
\nabla^2\left(\Delta U_1+\frac{1}{2}|\nabla U_1|^2-2U_2\right)\leq k^2I.
\]
Then any positive solution $\rho_t$ of the equation  $\dot\rho_t=\Delta\rho_t+\left<\nabla U_1,\nabla\rho_t\right>+U_2\rho_t$ satisfies the following estimate:
\[
\hess\log \rho_t+\frac{1}{2}\hess U_1\geq -\frac{k\coth(kt)}{2}I,
\]
where $\hess$ denotes the Hessian operator.
\end{thm}

Once again, if the underlying manifold is $\Real^n$, $U_1(x)=-\frac{k}{2}|x|^2$, and $U_2\equiv 0$, then
\[
\nabla^2\left(\Delta U_1+\frac{1}{2}|\nabla U_1|^2\right)= k^2I,
\]
and
\[
\hess\log \rho_t+\frac{1}{2}\hess U_1= -\frac{k}{2}\coth(kt)I.
\]
Therefore, the inequalities in Theorem \ref{newLiYauHam} are equalities in this case.

By setting $U_1\equiv 0$, $U_2\equiv 0$, and letting $k\to 0$, we recover

\begin{thm}\label{LiYauHam}(The Hamilton-Li-Yau estimate \cite{Ha1})
Assume that the sectional curvature of the underlying compact Riemannian manifold $M$ is non-negative and the Ricci curvature is parallel. Then any positive solution $\rho_t$ of the equation $\dot\rho_t=\Delta\rho_t$ satisfies the following estimate:
\[
\hess\log \rho_t\geq -\frac{1}{2t}I.
\]
\end{thm}

In \cite{Ha3}, Theorem \ref{LiYauHam} was used to prove a generalization of Huisken's monotonicity formula for the mean curvature flow \cite{Hu1}. More precisely, let $M$ be a $m$-dimensional sub-manifold of a $n$-dimensional Riemannian manifold $N$. Let $\varphi_t:M\to N$ be a family of immersions evolved according to the following equation
\begin{equation}\label{meanflow}
\dot\varphi_t=\meanvec_t(\varphi_t),
\end{equation}
where $\meanvec_t$ is the mean curvature vector of the sub-manifold $M_t:=\varphi_t(M)$.

\begin{thm}(Huisken's monotonicity formula \cite{Hu1, Ha3})
Assume that the sectional curvature of the underlying compact Riemannian manifold $N$ is non-negative and the Ricci curvature is parallel. Let $\varphi_t$ be a solution of (\ref{meanflow}) and let $\rho_t$ be a positive solution of the heat equation $\dot\rho_t=\bar\Delta \rho_t$ on $N$. Here $\bar \Delta$ denotes the Laplacian operator on $N$. Then
\[
\begin{split}
&\frac{d}{dt}\left((T-t)^{\frac{n-m}{2}}\int_{\varphi_t(M)}\rho_{T-t}\,d\mu_t\right)\\
&\leq -(T-t)^{\frac{n-m}{2}}\int_{\varphi_t(M)}\rho_{T-t}\,\left(\left|\frac{\nabla_t^\perp u_t}{u_t}-\meanvec_t\right|^2\right)d\mu_t,
\end{split}
\]
where $\mu_t$ is the Riemannian volume of $M_t$, $\bar\nabla u$ is the gradient of $u$ on $N$, and $\nabla_t^\perp u_t$ is the projection of $\bar\nabla u$ onto the normal bundle of $M_t$.

In particular, the quantity $(T-t)^{\frac{n-m}{2}}\int_{\varphi_t(M)}\rho_{T-t}\,d\mu_t$ is monotone.
\end{thm}

There is an analogue of this monotonicity formula in the setting of Theorem \ref{newLiYauHam}. In this case, the evolving hypersurfaces $M_t$ satisfy the following equation instead
\begin{equation}\label{newmeanflow}
\dot\varphi_t=\meanvec_t(\varphi_t)+\nabla_t^\perp U.
\end{equation}

We remark that the term $\meanvec_t(\varphi_t)+\nabla_t^\perp U$ is a generalization of mean curvature first appeared in \cite{Gro}. In particular, the equation (\ref{newmeanflow}) is the gradient flow of the weighted volume functional
\[
\int_{\varphi(M)} e^{-U} d\nu,
\]
where $\nu$ is the Riemannian volume on $\varphi(M)$ induced by the one on $N$.

Special cases of the equation was also studied in \cite{ScSm} and \cite{BoMi}.

\begin{thm}\label{newHuisken}(Generalized Huisken's monotonicity formula)
Assume that the sectional curvature of the underlying compact Riemannian manifold $N$ is non-negative and the Ricci curvature is parallel. Let $U:M\to\Real$ be a smooth function satisfying the following condition for some positive constant $k$:
\[
\bar \nabla^2\left(-\Delta U+\frac{1}{2}|\nabla U|^2\right)\leq k^2I,
\]
where $\bar\nabla^2$ is the Hessian operator on $N$.
Let $\varphi_t$ be a solution of (\ref{newmeanflow}) and let $\rho_t$ be a positive solution of the equation
\[
\dot\rho_t=\bar\Delta \rho_t+\left<\bar\nabla U,\bar\nabla \rho_t\right>+\rho_t\bar\Delta U
\]
on $N$. Then
\[
\begin{split}
&\frac{d}{dt}\left(\sinh^{\frac{n-m}{2}}(k(T-t))\int_{\varphi_t(M)}\rho_{T-t}\,d\mu_t\right)\\
&\leq -\sinh^{\frac{n-m}{2}}(k(T-t))\int_{\varphi_t(M)}\rho_{T-t}\,\left(\frac{1}{2}\Delta_t^\perp U+\left|\frac{\bar\nabla^\perp u_t}{u_t}-\meanvec_t\right|^2\right)d\mu_t,
\end{split}
\]
where $\Delta_t^\perp U$ is defined by $\Delta_t^\perp U=\sum_k\left<\bar\nabla_{\bn_k(t)}U,\bn_k(t)\right>$.
\end{thm}

As an immediate consequence, we have

\begin{cor}\label{newHuiskenCor}\label{lastcor}
Assume that the sectional curvature of the underlying compact Riemannian manifold $N$ is non-negative and the Ricci curvature is parallel. Let $U:M\to\Real$ be a smooth function satisfying the following condition for some constants $K$ and $k$ with $k>0$:
\[
\bar \nabla^2\left(-\Delta U+\frac{1}{2}|\nabla U|^2\right)\leq k^2I \quad \text{ and } \quad \bar\nabla^2U\geq KI.
\]
Let $\varphi_t$ be a solution of (\ref{newmeanflow}) and let $\rho_t$ be a positive solution of the equation
\[
\dot\rho_t=\bar\Delta \rho_t+\left<\bar\nabla U,\bar\nabla \rho_t\right>+\rho_t\bar\Delta U
\]
on $N$. Then
\[
\begin{split}
&\frac{d}{dt}\left(e^{-\frac{K(n-m)(T-t)}{2}}\sinh^{\frac{n-m}{2}}(k(T-t))\int_{\varphi_t(M)}\rho_{T-t}\,d\mu_t\right)\\
&\leq -e^{-\frac{K(n-m)(T-t)}{2}}\sinh^{\frac{n-m}{2}}(k(T-t))\int_{\varphi_t(M)}\rho_{T-t}\,\left(\left|\frac{\nabla_t \perp u_t}{u_t}-\meanvec_t\right|^2\right)d\mu_t.
\end{split}
\]
In particular, $e^{-\frac{K(n-m)(T-t)}{2}}\sinh^{\frac{n-m}{2}}(k(T-t))\int_{\varphi_t(M)}\rho_{T-t}\,d\mu_t$ is monotone.
\end{cor}

Remarkably, Corollary \ref{lastcor} is also sharp. In this case, we set $M=\Real^n$, $U=-\frac{k}{2}|x|^2$, and $K=-k$. Then
\[
\rho_t(x)=\left(\frac{2\pi(\exp(2tk)-1)}{k\exp(2tk)}\right)^{-n/2}\exp\left(\frac{-k|x|^2}{2(\exp(2tk)-1)}\right)\exp(knt)
\]
is a solution of the equation
\[
\dot\rho_t=\bar\Delta \rho_t+\left<\bar\nabla U,\bar\nabla \rho_t\right>+\rho_t\bar\Delta U=\bar\Delta \rho_t-k\left<x,\bar\nabla \rho_t\right>-kn\rho_t.
\]
It follows from the proof of Corollary \ref{lastcor} that all inequalities in the corollary are equalities in this case.

Assuming that the underlying manifold $M$ is compact, Theorem \ref{main} can be proved using the Bochner formula and the maximum principle. However, instead of the Bochner formula, we will prove a general result (Theorem \ref{main1} and \ref{main2}) using a moving frame argument motivated by the theory of optimal transportation (see \cite{Vi1}). This allows a more unified treatment for Theorem \ref{main} and \ref{newLiYauHam} under the compactness assumption. In section 3 and 4, we show that the above generalization of the Li-Yau estimate and its matrix analogue are simple consequences of Theorem \ref{main1} and \ref{main2}. In section 5, we give the proof of the generalized Huisken's monotonicity formula.

The Aronzon-B\'enilan estimate is a differential Harnack inequality for the porous medium equation
\[
\dot\rho_t=\Delta(\rho_t^m).
\]
In section 5, we will prove a generalization of Aronzon-B\'enilan estimate using Theorem \ref{main1} and \ref{main2}. We will prove sharp Laplace and Hessian type comparison theorems for the cost function (\ref{newcost}) in section 6. In section 7, a semigroup proof, in the spirit of \cite{BaLe}, of the generalized Li-Yau estimates will be discussed (again assuming $M$ is compact). In section 8, we give a proof of Theorem \ref{main} without any compactness assumption.

\smallskip

\section{Preliminaries}

In this section, we state and prove general results which will be used in the next few sections. For this, we will introduce some notations. Let $M$ be a $n$-dimensional compact manifold without boundary equipped with a Riemannian metric denoted by $\left<\cdot,\cdot\right>$ or $g$. The corresponding Riemann curvature tensor is denoted by $\Rm$. Let $F$ be a function on the space of all $n\times n$ matrices. We assume that $F$ is invariant under orthogonal changes of variables (i.e. $F(O^TAO)=F(A)$ for each orthogonal matrix $O$). For each linear map $W:T_xM\to T_xM$ of the tangent space $T_xM$ at a point $x$, we set $F(W)=F(\mathcal W)$, where $\mathcal W$ is the matrix with $ij$-th entry equal to $\left<W(v_i),v_j\right>$ and $\{v_1,...,v_n\}$ is an orthonormal frame at $x$. This is well-defined since $F$ is invariant under orthogonal changes of variables. Note that this condition is not needed or can be relaxed when the tangent bundle $TM$ of $M$ is parallelizable. For instance, when the manifold is the flat torus, this condition can be completely removed. Finally, if $u$, $v$, and $w$ are tangent vectors, then $u\otimes v$ denotes the linear map defined by $u\otimes v(w)=\left<v,w\right> u$.

The following is a generalization of the Li-Yau estimate \cite{LiYa}.

\begin{thm}\label{main1}
Assume that there is a non-negative function $b_t:M\to\Real$, a time dependent vector field $Y_t$ on a compact manifold $M$, and a fibre-preserving bundle homomorphism $W_t:TM\to TM$ of the tangent bundle $TM$ such that
\begin{enumerate}
\item $F'(A)(B^2)\geq k_1F(B)^2$ for some non-negative constant $k_1$,
\item $F'(\nabla X_t)(W_t+\Rm(\cdot,X_t)X_t)\geq k_3$ for some constant $k_3$,
\item $F'(\nabla X_t)(\nabla(\dot X_t+\nabla_{X_t} X_t)+W_t)+k_2F(\nabla X_t)^2\\
\leq F'(\nabla X_t)(b_t\nabla^2 (F(\nabla X_t))+\nabla (F(\nabla X_t))\otimes Y_t)$,
\item $k_1+k_2> 0$,
\end{enumerate}
Then
\[
F(\nabla X_t)\leq \frac{1}{k_1+k_2} a_{(k_1+k_2)k_3}(t),
\]
where
\[
a_K(t)= \begin{cases}
\sqrt{K}\cot(\sqrt{K}\,t) & \mbox{if $K>0$}\\
\frac{1}{t} & \mbox{if $K= 0$}\\
\sqrt{-K}\coth(\sqrt{-K}\,t) & \mbox{if $K<0$}.
\end{cases}
\]
\end{thm}

\begin{rem}
Note that the above theorem can be further generalized to include situation considered in \cite{BaLe} if $F$ is allowed to depend on $X_t$, not just $\nabla X_t$. However, we will not pursue this here.
\end{rem}

A matrix version of Li-Yau estimate was done by Hamilton \cite{Ha1}. The following is the corresponding matrix version of Theorem \ref{main1}.

\begin{thm}\label{main2}
Assume that there is a non-negative function $b_t:M\to\Real$, a time dependent vector field $Y_t$ on a compact manifold $M$, and a fibre-preserving bundle homomorphism $W_t:TM\to TM$ of the tangent bundle $TM$ such that
\begin{enumerate}
\item $w\mapsto \left<X_t,w\right>$ is a closed 1-form,
\item $W_t+\Rm(\cdot,X_t)X_t\geq k_3I$ for some constant $k_3$,
\item $\left<\nabla_v(\dot X_t+\nabla X_t(X_t)),v\right>+k_2\left<\nabla X_t(\nabla X_t(v)),v\right>+\left<W_tv,v\right>\\
\leq b_t\left<\Delta\nabla X_t(v),v\right>+\left<\nabla_{Y_t} \nabla_v X_t,v\right>$ for each eigenvector of the linear map $w\mapsto \nabla_w X_t$ with the largest eigenvalue,
\item $1+k_2> 0$,
\end{enumerate}
Then
\[
\nabla X_t\leq \frac{1}{1+k_2} a_{(1+k_2)k_3}(t)I.
\]
\end{thm}

As a consequence, we obtain the following estimate on the volume growth of a set under the flow of the vector field $X_t$ if $F=\tr$.

\begin{cor}\label{maincor1}
Under the assumptions of Theorem \ref{main1} with $F=\tr$,
\[
(b_{(k_1+k_2)k_3}(t))^{-\frac{1}{k_1+k_2}}\vol\,(\varphi_t(D))
\]
is a decreasing function of time $t$, where
\[
b_K(t)= \begin{cases}
\frac{1}{\sqrt K}\sin(\sqrt Kt) & \mbox{if $K>0$}\\
t & \mbox{if $K= 0$}\\
\frac{1}{\sqrt{-K}}\sinh(\sqrt{-K}t) & \mbox{if $K<0$}.
\end{cases}
\]
\end{cor}

The rest of this section is devoted to the proof of the above mentioned results.

\begin{proof}[Proof of Theorem \ref{main1}]
Let $\varphi_t$ be the one-parameter family of diffeomorphisms defined by the vector field $X_t$: $\dot\varphi_t=X_t(\varphi_t)$ and $\varphi_0(x)=x$. Let $\gamma(s)$ be a curve which start from $x$ with initial velocity $v$: $\gamma(0)=x$ and $\gamma'(0)=v$. Then
\[
\frac{D}{dt} d\varphi_t(v)=\frac{D}{ds}\frac{D}{dt}\varphi_t(\gamma(s))\Big|_{s=0}=\nabla_{d\varphi_t(v)}X_t.
\]

Let $v_1(0),...,v_n(0)$ be an orthonormal frame at a point $x$ and let $v_1(t),...,v_n(t)$ be the parallel transport of $v_1(0),...,v_n(0)$ along the path $\varphi_t(x)$. Let $A(t)$ be the matrix defined by
\[
d\varphi_t(v_i(0))=\sum_{j=1}^nA_{ij}(t)v_j(t).
\]

It follows that
\[
\sum_{j=1}^n\dot A_{ij}(t)v_j(t)=\sum_{j=1}^n A_{ij}(t)\nabla_{v_j(t)}X_t.
\]
Therefore, if $S_{ij}(t)=\left<\nabla_{v_i(t)}X_t,v_j(t)\right>$, then $S(t)=A(t)^{-1}\dot A(t)$ and we have
\begin{equation}\label{m1}
\begin{split}
\dot S(t)&=-A(t)^{-1}\dot A(t)A(t)^{-1}\dot A(t)+A(t)^{-1}\ddot A(t)\\
&=-S(t)^2+A(t)^{-1}\ddot A(t).
\end{split}
\end{equation}

On the other hand, if we differentiate the equation $\dot\varphi_t=X_t(\varphi_t)$, then we get
\[
\frac{D}{dt}\dot\varphi_t=\dot X_t(\varphi_t)+\nabla_{X_t}X_t(\varphi_t)
\]
and
\[
\frac{D}{ds}\frac{D}{dt} \dot\varphi_t(\gamma(s))\Big|_{s=0}=\nabla_{d\varphi_t(v)}\left(\dot X_t+\nabla_{X_t}X_t\right).
\]

By the definition of the Riemann curvature tensor $\Rm$, it follows that
\[
\frac{D^2}{dt^2} d\varphi_t(v)+\Rm(d\varphi_t(v),X_t(\varphi_t))X_t(\varphi_t)=\nabla_{d\varphi_t(v)}\left(\dot X_t+\nabla_{X_t}X_t\right).
\]

Therefore, by the definition of the matrix $A(t)$, the following holds
\[
\ddot A(t)+A(t)(R(t)-M(t))=0,
\]
where
\[
R_{ij}(t)=\left<\Rm(v_i(t),X_t(\varphi_t(x)))X_t(\varphi_t(x)),v_j(t)\right>
\]
and
\[
M_{ij}(t)=\left<\nabla_{v_i(t)}\left(\dot X_t+\nabla_{X_t}X_t\right),v_j(t)\right>_{\varphi_t(x)}.
\]

By combining this with (\ref{m1}), we obtain
\begin{equation}\label{Riccati}
\dot S(t)+S(t)^2+R(t)=M(t).
\end{equation}

It follows that
\begin{equation}\label{m2}
\begin{split}
\frac{d}{dt} F(S(t))+ F'(S(t))(S(t)^2+R(t))=F'(S(t))(M(t)).
\end{split}
\end{equation}

Let $t_0$ be the first time where $F(\nabla X_{t_0}(\varphi_{t_0}(x)))=k a_K(t_0)$ for some point $x$, where $k>0$. By assumption, we have
\begin{equation}\label{m5}
\begin{split}
&F'(\nabla X_{t_0})(\nabla(\dot X_{t_0}+\nabla_{X_{t_0}} X_{t_0})+W_t)+k_2F(\nabla X_{t_0})^2\leq 0
\end{split}
 \end{equation}
at $\varphi_{t_0}(x)$.

In the matrix notation, we have
\[
F'(S(t_0))(M(t_0)+\mathcal W(t_0))+k_2F(S(t_0))^2\leq 0,
\]
where $\mathcal W(t_0)$ be the matrix with $ij$-th entry equal to $\left<W_{t}(v_i(t)),v_j(t)\right>$.

By combining this with (\ref{m2}) and using the assumptions, we obtain
\[
\begin{split}
&\frac{d}{dt} F(S(t_0))+ k_1F(S(t_0))^2+ k_3\\
&\leq \frac{d}{dt} F(S(t_0))+ k_1F(S(t_0))^2+ F'(S(t_0))(R(t_0)+\mathcal W(t_0))\\
&\leq \frac{d}{dt} F(S(t_0))+ F'(S(t_0)^2)+ F'(S(t_0))(R(t_0)+\mathcal W(t_0))\\
&= F'(S(t_0))(M(t_0)+\mathcal W(t_0))\\
&\leq -k_2F(S(t_0))^2.
\end{split}
\]

By the definition of $t_0$, we have $k a_K(t_0)= F(S(t_0))$ and $k\dot a_K(t_0)\leq\frac{d}{dt} F(S(t_0))$. Therefore, the above inequality becomes
\[
k\dot a_K(t_0)+ (k_1+k_2)k^2a_K(t_0)^2+ k_3\leq 0.
\]

Since $a_K$ satisfies
\begin{equation}\label{Riccati0}
\dot a_K+a_K^2+K=0,
\end{equation}
it follows that
\[
k((k_1+k_2)k-1)a(t_0)^2+ k_3-kK\leq 0.
\]

Therefore, we obtain a contradiction if $k=\frac{1}{k_1+k_2}$ and $K < (k_1+k_2)k_3$. Hence
\[
F(\nabla X_t)<\frac{1}{k_1+k_2} a_K(t)
\]
for all $K<(k_1+k_2)k_3$. By letting $K\to (k_1+k_2)k_3$, we obtain
\[
F(\nabla X_t)\leq \frac{1}{k_1+k_2} a_{(k_1+k_2)k_3}(t).
\]
\end{proof}

\begin{proof}[Proof of Theorem \ref{main2}]
Here, we use the same notations as in the proof of Theorem \ref{main1}. By assumption the one-form $v\mapsto \left<X_t,v\right>$ is closed. This is equivalent to $\left<\nabla_vX_t,w\right>=\left<v,\nabla_wX_t\right>$. It follows that the matrices $S(t)$ are all symmetric. Let $t_0$ be the first time such that there is a point $x$ and a unit tangent vector $v$ in the tangent space $T_{\varphi_t(x)}M$ at $\varphi_t(x)$ such that $\left<\nabla_v X_{t_0},v\right>=\left<S(t_0)v,v\right>=ka_K(t_0)$. Here $v$ denotes both the vector $v$ and its matrix representation with respect to the orthonormal frame $v_1(t),...,v_n(t)$. In particular, $ka_K(t_0)$ is the largest eigenvalue of $S(t_0)$ with eigenvector $v$. By parallel translating along geodesics, we extend $v$ to a vector field still denoted by $v$. It follows that $\nabla v=0$ and $\Delta v=0$. Therefore, the following holds by assumption
\[
\begin{split}
&\left<\nabla_v(\dot X_{t_0}+\nabla X_{t_0}(X_{t_0})),v\right>+k_2\left<\nabla X_{t_0}(\nabla X_{t_0}(v)),v\right>+\left<W_{t_0}v,v\right>\\
&\leq b_{t_0}\left<\Delta\nabla X_{t_0}(v),v\right>+\left<\nabla_{Y_{t_0}} \nabla_v X_{t_0},v\right>\\
&\leq b_{t_0}\Delta\left<\nabla X_{t_0}(v),v\right>+\nabla_{Y_{t_0}} \left<\nabla_v X_{t_0},v\right>\leq 0.
\end{split}
 \]

In terms of the matrix notations, the above inequality becomes
\[
\left<(M(t_0)+k_2S(t_0)^2+\mathcal W(t_0))v,v\right>\leq 0.
\]

This, together with (\ref{Riccati}) and (\ref{Riccati0}), gives
\[
\begin{split}
0&\leq \frac{d}{dt} \left(\left<S(t)v,v\right>-ka_K(t)\right)\Big|_{t=t_0}\\
&=-\left<S(t_0)^2v,v\right>+\left<(M(t_0)-R(t_0))v,v\right>+ka_K(t_0)^2+kK\\
&\leq -(1+k_2)\left<S(t_0)^2v,v\right>-\left<(\mathcal W(t_0)+R(t_0))v,v\right>+ka_K(t_0)^2+kK.
\end{split}
\]

By assumption, $\mathcal W(t)+R(t)\geq k_3I$. It follows that
\[
k(1-(1+k_2)k)a_K(t_0)^2+kK\geq k_3.
\]

Therefore, we obtain a contradiction if $k=\frac{1}{1+k_2}$ and $K<k_3(1+k_2)$. It follows that
\[
\nabla X_t\leq \frac{a_{k_3(1+k_2)}(t)}{1+k_2}I.
\]
\end{proof}

\begin{proof}[Proof of Corollary \ref{maincor1}]
If $F(\nabla X)\geq \tr(\nabla X)$, then
\[
\frac{d}{dt}\log\det A(t)\leq F(\nabla X_{\varphi_t(x)}).
\]
It follows that
\[
\frac{\det(d\varphi_{t_1})}{\det(d\varphi_{t_0})}\leq \exp(\int_{t_0}^{t_1}F(\nabla X_{\varphi_t(x)})dt)\leq \frac{b^{\frac{1}{k_1+k_2}}_{(k_1+k_2)k_3}(t_1)}{b^{\frac{1}{k_1+k_2}}_{(k_1+k_2)k_3}(t_0)},
\]
where
\[
b_K(t)= \begin{cases}
\frac{1}{\sqrt K}\sin(\sqrt Kt) & \mbox{if $K>0$}\\
t & \mbox{if $K= 0$}\\
\frac{1}{\sqrt{-K}}\sinh(\sqrt{-K}t) & \mbox{if $K<0$}.
\end{cases}
\]

\end{proof}

\smallskip

\section{A generalization of the Li-Yau estimate: the case on compact manifolds}

In this section, we prove the following generalization of the Li-Yau estimate.

\begin{thm}\label{main3}
Assume that the Ricci curvature of the underlying compact Riemannian manifold $M$ is non-negative. Let $U_1$ and $U_2$ be two functions on $M$ satisfying
\[
\Delta\left(-\Delta U_1-\frac{1}{2}|\nabla U_1|^2+2U_2\right)\geq k_3.
\]
Then any positive solution $\rho_t$ of the equation
\begin{equation}\label{2nd}
\dot\rho_t=\Delta\rho_t+\left<\nabla\rho_t,\nabla U_1\right>+U_2\rho_t.
\end{equation}
satisfies
\[
2\Delta\log\rho_t+\Delta U_1 \geq -n a_{\frac{k_3}{n}}(t).
\]
\end{thm}

By integrating the above generalization of Li-Yau estimate, one obtains a Harnack inequality. For this, we need to consider the following functional
\[
\int_{s_0}^{s_1}\frac{1}{2}|\dot\gamma(\tau)|^2+V(\gamma(\tau))d\tau,
\]
where $\gamma:[s_0,s_1]\to M$ and $V=\Delta U_1+\frac{1}{2}|\nabla U_1|^2-2U_2$.

Let $c_{s_0,s_1}$ be the corresponding cost function defined by
\begin{equation}\label{costa}
c_{s_0,s_1}(x,y)=\inf\int_{s_0}^{s_1}\frac{1}{2}|\dot\gamma(\tau)|^2+V(\gamma(\tau))d\tau,
\end{equation}
where the infimum is taken over all paths $\gamma$ satisfying $\gamma(s_0)=x$ and $\gamma(s_1)=y$.

\begin{cor}\label{maincor2}
Under the assumptions of Theorem \ref{main3}, the following holds
\[
\frac{\rho_{s_1}(y)}{\rho_{s_0}(x)}\geq \left(\frac{b_{\frac{k_3}{n}}(s_1)}{b_{\frac{k_3}{n}}(s_0)}\right)^{-\frac{n}{2}}\exp\left(-\frac{1}{2}\left(c_{s_0,s_1}(x,y)+U_1(y)-U_1(x)\right)\right)
\]
\end{cor}

If we let $\rho_t$ be the fundamental solution $p_t(x,y)$ of the equation (\ref{2nd}) and let $s\to 0$ in Corollary \ref{maincor2}, then we obtain the following generalization of Cheeger-Yau estimate \cite{ChYa}.

\begin{cor}\label{maincor3}
Let $p_t$ be the fundamental solution of the equation (\ref{2nd}). Under the assumptions of Theorem \ref{main3}, the following holds
\[
p_t(x,y)\geq \left(4\pi b_{\frac{k_3}{n}}(t)\right)^{-\frac{n}{2}}\exp\left(-\frac{1}{2}\left(c_{0,t}(x,y)+U_1(y)-U_1(x)\right)\right).
\]
\end{cor}

Finally, we will show that the equality case in Corollary \ref{maincor3} is achieved by (\ref{fund}). More precisely,

\begin{thm}\label{maincor4}
Let $\rho_t$ be defined by (\ref{fund}), $U_1(x)=-\frac{k}{2}|x|^2$, and $U_2\equiv 0$. Then
\[
p_t(0,x)= \exp\left(-\frac{1}{2}\left(c_{0,t}(0,x)+U_1(x)-U_1(0)\right)\right)\left(4\pi b_{-k^2}(t)\right)^{-\frac{n}{2}}.
\]
\end{thm}

\begin{proof}[Proof of Theorem \ref{main3}]
If we specialize Theorem \ref{main1} to the case where $F=\tr$ and $X_t=\nabla h_t$, then the assumptions of Theorem \ref{main1} are satisfied if $k_1=\frac{1}{n}$,  $\tr(W_t)+\ric(X_t,X_t)\geq k_3$, and
\begin{equation}\label{m3}
\begin{split}
&\Delta\left( \dot h_t+\frac{1}{2}|\nabla h_t|^2\right)+k_2(\Delta h_t)^2+\tr(W_t)\\
&\leq b_t\Delta\Delta h_t+\left<\nabla \Delta h_t, Y_t\right>.
\end{split}
 \end{equation}

Let $h_t=-2\log \rho_t-U_1$. Then the following holds
\[
\dot h_t+\frac{1}{2}|\nabla h_t|^2=\Delta h_t+\Delta U_1+\frac{1}{2}|\nabla U_1|^2-2U_2.
\]

Therefore, under the assumptions of the theorem, (\ref{m3}) holds with $k_2=0$ and $b_t\equiv 1$. Hence, the result follows from Theorem \ref{main1}.
\end{proof}

\begin{proof}[Proof of Corollary \ref{maincor2}]
Let $\gamma$ be a minimizer of (\ref{costa}) which satisfies $\gamma(s_0)=x_0$ and $\gamma(s_1)=x_1$. Using the notations in the proof of Theorem \ref{main3}, we have
\[
\begin{split}
&\frac{d}{dt} h_t(\gamma(t))-\frac{1}{2}|\dot\gamma(t)|^2\\
&\leq\frac{d}{dt} h_t(\gamma(t))-\left<\nabla h_t(\gamma(t)),\dot\gamma(t)\right>+\frac{1}{2}|\nabla h_t|^2_{\gamma(t)}\\
&=\Delta h_t(\gamma(t))+\Delta U_1(\gamma(t))+\frac{1}{2}|\nabla U_1|^2_{\gamma(t)}-2U_2(\gamma(t))\\
&\leq na_{\frac{k_3}{n}}(t)+V(\gamma(t)).
\end{split}
\]
In the last inequality above, we have used Theorem \ref{main3}.

By integrating the above inequality and noting that $\dot b_K=b_Ka_K$, we obtain
\[
h_{s_1}(x_1)-h_{s_0}(x_0)\leq c_{s_0,s_1}(x_0,x_1)+n\log(b_{\frac{k_3}{n}}(s_1))-n\log(b_{\frac{k_3}{n}}(s_0)).
\]

By taking exponential of the above inequality, the result follows.
\end{proof}

\begin{proof}[Proof of Corollary \ref{maincor3}]
By Corollary \ref{maincor2}, we have
\[
\frac{p_t(x,y)}{p_s(x,x)}\geq \exp\left(-\frac{1}{2}\left(c_{s,t}(x,y)+U_1(y)-U_1(x)\right)\right)\left(\frac{b_{\frac{k_3}{n}}(t)}{b_{\frac{k_3}{n}}(s)}\right)^{-\frac{n}{2}}.
\]

Since $\lim_{s\to 0}(4\pi s)^{n/2}p_s(x,x)=1$ (see \cite{Se,Gi}), the above inequality gives
\[
p_t(x,y)\geq \exp\left(-\frac{1}{2}\left(c_{0,t}(x,y)+U_1(y)-U_1(x)\right)\right)\left(4\pi b_{\frac{k_3}{n}}(t)\right)^{-\frac{n}{2}}
\]
as claimed.
\end{proof}

\begin{proof}[Proof of Theorem \ref{maincor4}]
In this special case, the cost function (\ref{costa}) is given by
\begin{equation}\label{m4}
c_{0,t}(0,y)=\inf\int_{0}^{t}\frac{1}{2}|\dot\gamma(s)|^2+V(\gamma(s))ds,
\end{equation}
where $V(x)=-kn+\frac{1}{2}k^2|x|^2$ and the infimum is taken over all paths $\gamma$ satisfying $\gamma(0)=0$ and $\gamma(t)=y$.

If $x(\cdot)$ is a minimizer of the above infimum, then it satisfies the following equations (see \cite{Ev})
\[
\dot x=p,\quad \quad \dot p=k^2 x.
\]

Since $x(0)=0$ and $x(t)=y$, it follows that
\[
x(s)=\frac{\sinh(ks)}{\sinh(kt)}\, x(t).
\]
If we substitute this back into (\ref{m4}), then we obtain
\[
c_{0,t}(0,y)=\frac{k|y|^2\coth(kt)}{2}-knt.
\]

A computation shows that
\[
\begin{split}
p_t(0,y)&= \exp\left(-\frac{1}{2}\left(c_{0,t}(0,y)-\frac{k}{2}|y|^2\right)\right)\left(4\pi b_{-k^2}(t)\right)^{-\frac{n}{2}}\\
&=\exp\left(\frac{-k|y|^2}{2(\exp(2kt)-1)}\right)\left(\frac{2\pi (\exp(2tk)-1)}{k\exp(2kt)}\right)^{-\frac{n}{2}}
\end{split}
\]
as claimed.
\end{proof}

\smallskip

\section{A generalization of Hamilton's matrix Li-Yau estimate}

In this section, we show that the following generalization of Hamilton-Li-Yau estimate is a consequence of Theorem \ref{main2}.

\begin{thm}\label{main4}
Assume that the sectional curvature of the underlying compact Riemannian manifold $M$ is non-negative and the Ricci curvature is parallel. Let $U_1$ and $U_2$ be two functions on $M$ satisfying
\[
-\nabla^2\left(\Delta U_1+\frac{1}{2}|\nabla U_1|^2-2U_2\right)\geq k_3I,
\]
Then any solution $\rho_t$ of the equation (\ref{2nd}) satisfies
\[
-2\hess\log \rho_t-\hess U_1\leq a_{k_3}(t)I.
\]
\end{thm}

\begin{proof}
We need the following lemma.

\begin{lem}
Assume that the sectional curvature of a Riemannian manifold is non-negative at a point $x$ and the Ricci curvature $\ric$ satisfies $\nabla\ric_x=0$. Then, for any smooth function $f$, the following holds
\[
\begin{split}
&\Delta(\nabla_vdf(v))(x) \geq \left<\nabla_v\nabla\Delta f,v\right>_x.
\end{split}
\]
Here we consider the Hessian $\nabla df$ of $f$ as a self-adjoint operator on $T_xM$. The vector field $v$ is defined as an eigenvector of the operator $\nabla df$ at $x$ corresponding to the largest eigenvalue and it is extended to a neighborhood of $x$ by parallel translation along geodesics.
\end{lem}

\begin{proof}
Let $e_1,...,e_n$ be an orthonormal frame at $x$ and let us extend them to vector fields defined locally near $x$ by parallel translation along geodesics. It follows that $\nabla v(x)=0$ and $\nabla_{e_i}\nabla_{e_i} v(x)=0$ (throughout this proof we sum over repeated indices without mentioning). Therefore,
\[
\Delta(\nabla_vdf(v))=\nabla_{e_i}\nabla_{e_i}\nabla_vdf(v).
\]

Let $\alpha$ be a $(0,1)$-tensor and $\beta$ be a $(0,2)$-tensor. By Ricci identity, we have
\begin{enumerate}
\item $\nabla_{v_1}\nabla_{v_2}\alpha(v_3)=\nabla_{v_2}\nabla_{v_1}\alpha(v_3)-\alpha(\Rm(v_1,v_2)v_3)$,
\item $\nabla_{v_4}\nabla_{v_1}\nabla_{v_2}\alpha(v_3)\\
=\nabla_{v_4}\nabla_{v_2}\nabla_{v_1}\alpha(v_3)-\nabla_{v_4}\alpha(\Rm(v_1,v_2)v_3)-\alpha(\nabla_{v_4}\Rm(v_1,v_2)v_3)$,
\item  $\nabla_{v_1}\nabla_{v_2}\beta(v_3,v_4)\\
=\nabla_{v_2}\nabla_{v_1}\beta(v_3,v_4)-\beta(\Rm(v_1,v_2)v_3,v_4)-\beta(v_3,\Rm(v_1,v_2) v_4)$.
\end{enumerate}
Here, for instance, $\nabla_{v_4}\nabla_{v_1}\nabla_{v_2}\alpha(v_3)$ denotes
\[
\nabla(\nabla(\nabla\alpha))(v_4,v_1,v_2, v_3).
\]

It follows that
\[
\begin{split}
\Delta(\nabla_vdf(v))
&=\nabla_{e_i}(\nabla_{v}\nabla_{e_i}df(v)-df(\Rm(e_i,v)v))\\
&=\nabla_{e_i}\nabla_{v}\nabla_{v}df(e_i)-\nabla_{e_i}df(\Rm(e_i,v)v)\\
&\quad -df(\nabla_{e_i}\Rm(e_i,v)v)\\
&=\nabla_v\nabla_{e_i}\nabla_{v}df(e_i)-\nabla_{e_i} df(\Rm(e_i,v)v)\\
&\quad -\nabla_v df(\Rm(e_i,v)e_i)-\nabla_{e_i}df(\Rm(e_i,v)v)\\
&\quad  -df(\nabla_{e_i}\Rm(e_i,v)v)\\
&=\nabla_v\nabla_{v}\nabla_{e_i}df(e_i)-\nabla_{v}df(\Rm(e_i,v)e_i)\\
&\quad -df(\nabla_{v}\Rm(e_i,v)e_i) -\nabla_{e_i} df(\Rm(e_i,v)v)\\
&\quad -\nabla_v df(\Rm(e_i,v)e_i)-\nabla_{e_i}df(\Rm(e_i,v)v)\\
&\quad  -df(\nabla_{e_i}\Rm(e_i,v)v)\\
&=\left<\nabla_v\nabla\Delta f,v\right>-2\nabla_{v}df(\Rm(e_i,v)e_i)\\
&\quad -df(\nabla_{v}\Rm(e_i,v)e_i) -2\nabla_{e_i} df(\Rm(e_i,v)v)\\
&\quad  -df(\nabla_{e_i}\Rm(e_i,v)v).
\end{split}
\]

Since the Ricci curvature is parallel, we have, by the contracted Bianchi identity,
\[
\begin{split}
&\Delta(\nabla_vdf(v))=\left<\nabla_v\nabla\Delta f,v\right>-2\nabla_{v}df(\Rm(e_i,v)e_i) -2\nabla_{e_i} df(\Rm(e_i,v)v).
\end{split}
\]

If $e_i$ is an eigenvector of the hessian of $f$ with eigenvalue $\lambda_i$ and $v$ is an eigenvector of the hessian of $f$ with the largest eigenvalue $\lambda$, then
\[
\begin{split}
&\Delta(\nabla_vdf(v))\\
&=\left<\nabla_v\nabla\Delta f,v\right>+2\lambda \ric(v,v) -2\lambda_i\left<e_i,\Rm(e_i,v)v\right>\\
&\geq \left<\nabla_v\nabla\Delta f,v\right>.
\end{split}
\]
Here we use the assumption that the sectional curvature is non-negative.
\end{proof}

When $X_t=\nabla h_t$, the conditions become $\mathcal W(t)+R(t)\geq k_3I$ and
\[
\begin{split}
&\left<\nabla_v\nabla\left(\dot h_t+\frac{1}{2}|\nabla h_t|^2\right),v\right>+k_2\left<(\nabla^2 h_t)^2v,v\right>+\left<W_tv,v\right>\\
&\leq b_t\left<\Delta\hess h_t(v),v\right>+\left<\nabla_{Y_t} \hess h_t(v),v\right>
\end{split}
 \]
for each eigenvector $v$ of the symmetric operator $\hess h_t$ with the largest eigenvalue.

Recall that if $\rho_t$ is a positive solution of the equation
\[
\dot\rho_t=\Delta\rho_t+\left<\nabla\rho_t,\nabla U_1\right>+U_2\rho_t,
\]
then $h_t=-2\log \rho_t-U_1$ satisfies
\[
\dot h_t+\frac{1}{2}|\nabla h_t|^2=\Delta h_t+\Delta U_1+\frac{1}{2}|\nabla U_1|^2-2U_2.
\]

It follows that
\[
\nabla^2\left(\dot h_t+\frac{1}{2}|\nabla h_t|^2\right)+W_t=\nabla^2\Delta h_t,
\]
where $W_t=-\nabla^2\left(\Delta U_1+\frac{1}{2}|\nabla U_1|^2-2U_2\right)$.

Therefore, if we assume that the Ricci curvature is parallel, the sectional curvature is non-negative, and
$W_t\geq k_3 I$, then
\[
\begin{split}
&\left<\nabla^2\left(\dot h_t+\frac{1}{2}|\nabla h_t|^2\right)(v),v\right>+\left<W_t(v),v\right>\leq \left<\Delta \hess h_t(v),v\right>.
\end{split}
 \]

It follows that
\[
\hess h_t=-2\hess\log \rho_t-\hess U_1\leq a_{k_3}(t)I.
\]
\end{proof}

\smallskip

\section{A generalization of Huisken's monotonicity formula}

This section is devoted to the proof of Theorem \ref{main}. First, let us recall the notations used. Let $M$ be a submanifold of dimension $m$ in a Riemannian manifold $N$ of dimension $n$. The mean curvature flow is a family of immersions $\varphi_t:M\to N$ which satisfy
\[
\dot\varphi_t=\meanvec_t(\varphi_t)+\nabla_t^\perp U(\varphi_t),
\]
where $\meanvec_t$ is the mean curvature vector of $M_t:=\varphi_t(M)$, $\bar\nabla U$ denotes the gradient of $U$ with respect to the Riemannian metric on $N$, and $\nabla_t^\perp U$ is the projection of $\bar\nabla_t U$ onto the normal bundle of $M_t$. We also introduce the following notation for the part of the Laplacian in the normal bundle $\Delta_t^\perp U=\sum_k\left<\bar\nabla_{\bn_k}\bar\nabla U,\bn_k\right>$.

\begin{thm}\label{genHui}
Assume that the sectional curvature of the underlying compact Riemannian manifold $N$ is non-negative and the Ricci curvature is parallel. Let $U:M\to\Real$ be a smooth function satisfying the following condition for some positive constant $k$:
\[
\nabla^2\left(\Delta U -\frac{1}{2}|\nabla U|^2\right)\geq k_3I,
\]
Let $\varphi_t$ be a solution of (\ref{newmeanflow}) and let $\rho_t$ be a positive solution of the equation
\[
\dot\rho_t=\bar\Delta \rho_t+\left<\bar\nabla U,\bar\nabla \rho_t\right>+\rho_t\bar\Delta U
\]
on $N$. Then
\[
\begin{split}
&\frac{d}{dt}\left(b_{k_3}(T-t)^{\frac{n-m}{2}}\int_{\varphi_t(M)}u_t\,d\mu_t\right)\\
&\leq -b_{k_3}(T-t)^{\frac{n-m}{2}}\int_{\varphi_t(M)}u_t\left(\frac{1}{2}\Delta_t^\perp U+\left|\frac{\nabla_t^\perp u_t}{u_t}-\meanvec_t\right|^2\right) d\mu_t.
\end{split}
\]
\end{thm}

The rest of this section is devoted to the proof of the above theorem. Next, we pick a convenient moving frame along $\varphi_t$.

\begin{lem}\label{frame}
Let $\sigma(\cdot)$ be a path in $N$ such that $\sigma(t)$ is contained in $M_t:=\varphi_t(M)$. Then there is a family of orthonormal frames
\[
\bn_1(\psi_t),..., \bn_{n-m}(\psi_t), v_1(t),...,v_m(t)
\]
defined along $\sigma(\cdot)$ such that
\begin{enumerate}
\item $v_1(t),...,v_m(t)$ are contained in the tangent bundle $TM_t$ of $M_t$,
\item $\bn_1(t),...,\bn_{n-m}(t)$ are in the normal bundle $TM_t^\perp$ of $M_t$,
\item $\dot v_1(t),...,\dot v_m(t)$ are in $TM_t^\perp$,
\item $\dot \bn_1(t), ..., \dot \bn_{n-m}(t)$ are in $TM_t$.
\end{enumerate}
Here $\dot v_i(t)$ denotes the covariant derivative of $v_i(t)$ with respect to the Riemannian metric $\left<\cdot,\cdot\right>$ of $N$.

Moreover, if $\tilde\bn_1(t),...,\tilde\bn_{n-m}(t),\tilde v_1(t),...,\tilde v_m(t)$ is another such family, then there are orthogonal matrices $O^{(1)}$ and $O^{(2)}$ (independent of time) of size $(n-m)\times(n-m)$ and $m\times m$, respectively, such that
\[
\tilde \bn_i(t)=\sum_{j=1}^{n-m} O_{ij}^{(1)}\bn_j(t)\text{ and }\tilde v_i(t)=\sum_{j=1}^m O_{ij}^{(2)}v_j(t).
\]
\end{lem}

The proof of Lemma \ref{frame} and that of \cite[Lemma 3.1]{Le} is completely analogous and is therefore omitted. From now on, we call any orthonormal moving frame which satisfies the conditions in Lemma \ref{frame} a parallel adapted frame along $\sigma(\cdot)$.

Let $\bn_t$ be a normal vector in $TM_t^\perp$ and let $\Sh_t^{\bn_t}:TM_t\to TM_t$ be the shape operator of the submanifold $M_t$ defined by
\[
\left<\Sh_t^{\bn_t}(v_1),v_2\right>=-\left<\bar\nabla_{v_1}\bn_t,v_2\right>.
\]
Here $\bar\nabla$ denotes the Levi-Civita connection on $N$.

Recall that the mean curvature vector $\meanvec_t$ of $M_t$ is given by
\[
\meanvec_t=\sum_{i,j}\left<\Sh_t^{\bn_i(t)}(v_j(t)),v_j(t)\right>\bn_i(t).
\]

\begin{lem}\label{paflem2}
Let $\bn_1(t),...,\bn_{n-m}(t),v_1(t),...,v_m(t)$ be a parallel adapted frame along $\varphi_t(x)$, where $\varphi_t$ satisfies the following equation
\[
\dot\varphi_t=\sum_i F_i(t,\varphi_t)\bn_i(t).
\]

Let $A(t)$ and $G^k(t)$ be families of matrices defined by
\[
d\varphi_t(v_i(0))=\sum_jA_{ij}(t)v_j(t) \text{ and } G^k_{ij}(t)=\left<\Sh_t^{\bn_k(t)}(v_i(t)),v_j(t)\right>,
\]
respectively. Then
\[
\dot A(t)=-\sum_k F_k(t,\varphi_t)A(t)G^k(t),
\]
where $\nabla_t$ is the gradient with respect to the induced metric on $M_t$.
\end{lem}

\begin{proof}
Let $\gamma(s)$ be a curve in $M$ such that $\frac{d}{ds}\gamma(s)\Big|_{s=0}=v_i(0)$. Then
\[
\begin{split}
&\covar t d\varphi_t(v_i(0))=\sum_j\left(\dot A_{ij}(t)v_j(t)+A_{ij}(t)\dot v_j(t)\right).
\end{split}
\]

On the other hand, we have
\[
\begin{split}
&\covar t d\varphi_t(v_i(0))=\covar s\dot\varphi_t(\gamma(s))\Big|_{s=0}\\
&=\sum_k\left(\left<\nabla_t F_k(t,\varphi_t),d\varphi_t(v_i(0))\right>\bn_k(t)+ F_k(t,\varphi_t)\bar\nabla_{d\varphi_t(v_i(0))}\bn_k(t)\right).
\end{split}
\]
It follows that
\[
\dot A_{ij}(t)=-\sum_{l,k}A_{il}(t)F_k(t,\varphi_t)\left<\Sh_t^{\bn_k(t)}v_l(t),v_j(t)\right>.
\]
\end{proof}

\begin{proof}[Proof of Theorem \ref{genHui}]
Let $\rho_t$ be the density of $\varphi_t^*\mu_t$ with respect to $\mu_0$: $\rho_t\mu_0=\varphi_t^*\mu_t$. Let $\bn_1(t),...,\bn_{n-m}(t), v_1(t),...,v_m(t)$ be a parallel adapted frame along the path $\varphi_t(x)$ and let $A(t)$ be the family of matrices defined by
\[
d\varphi_t(v_i(0))=\sum_{j=1}^nA_{ij}(t)v_j(t).
\]
Then $\rho_t=\det A(t)$ and we have
\[
\begin{split}
&\frac{d}{dt}\int_{\varphi_t(M)}u_t\,d\mu_t=\frac{d}{dt}\int_{M}u_t(\varphi_t)\det A(t)\,d\mu_0\\
&=\int_{M}\Big[\dot u_t(\varphi_t)+u_t(\varphi_t)\tr(A(t)^{-1}\dot A(t))+\left<\bar\nabla u_t,\dot \varphi_t\right>_{\varphi_t}\Big]\det A(t)\,d\mu_0\\
&=\int_{\varphi_t(M)}\left(\dot u_t+\sum_kF_k\left<\bar\nabla u_t,\bn_k(t)\right>-u_t\sum_kF_k\tr(G^k(t))\right)d\mu_t.
\end{split}
\]

Let $u_t=\rho_{T-t}$. Then we have, by assumptions, $F_k(t,\cdot)=\tr(G^k(t))+\left<\bar\nabla U,\bn_k(t)\right>$ and $\dot u_t=-\bar\Delta u_t-\left<\bar\nabla U,\bar\nabla u_t\right> -(\bar \Delta U) u_t$. Then the above equation becomes
\[
\begin{split}
&\frac{d}{dt}\int_{\varphi_t(M)}u_t\,d\mu_t=\int_{\varphi_t(M)}\Big(-\bar\Delta u_t-\left<\nabla U,\nabla u_t\right>\\
& -(\bar \Delta U) u_t-u_t|\meanvec_t|^2+\left<\nabla_t^\perp u_t,\meanvec_t\right>-u_t\left<\nabla_t^\perp U,\meanvec_t\right>\Big)d\mu_t,
\end{split}
\]
where $\nabla_t^\perp u$ is the projection of $\bar\nabla u$ onto the normal bundle of $M_t$.

A simple calculation shows that
\[
\begin{split}
\Delta u&=\sum_{i=1}^n\left<\bar\nabla_{v_i}(\bar\nabla u-\sum_k\left<\bn_k,\bar\nabla u\right>\bn_k),v_i\right>\\
&=\bar\Delta u-\sum_k\left<\bar\nabla_{\bn_k}\bar\nabla u,\bn_k\right>+\sum_k\left<\bn_k,\bar\nabla u\right>\tr(G^k(t))\\
&=\bar\Delta u-\Delta_t^\perp u+\left<\meanvec,\nabla_t^\perp u\right>,
\end{split}
\]

Therefore, we have
\[
\begin{split}
&\frac{d}{dt}\int_{\varphi_t(M)}u_t\,d\mu_t=\int_{\varphi_t(M)}\Big(-\Delta u_t-\Delta_t^\perp u_t-u_t\Delta_t^\perp U\\
&-\left<\nabla U,\nabla u_t\right> -\Delta U u_t-u_t|\meanvec_t|^2+2\left<\nabla_t^\perp u_t,\meanvec_t\right>\Big)d\mu_t\\
&=\int_{\varphi_t(M)}\Big(-\Delta_t^\perp u_t-u_t\Delta_t^\perp U-u_t|\meanvec_t|^2\\
&+2u_t\left<\frac{\nabla_t^\perp u_t}{u_t},\meanvec\right>-u_t\left|\frac{\nabla_t^\perp u_t}{u_t}\right|^2+u_t\left|\frac{\nabla_t^\perp u_t}{u_t}\right|^2\Big)d\mu_t\\
&=-\int_{\varphi_t(M)}u_t\left(\Delta_t^\perp \log u_t+\Delta_t^\perp U+\left|\frac{\nabla_t^\perp u_t}{u_t}-\meanvec_t\right|^2\right) d\mu_t.
\end{split}
\]

\[
-\Delta_t^\perp\log \rho_{T-t}-\frac{1}{2}\Delta_t^\perp U_1\leq \frac{n-m}{2}a_{k_3}(T-t).
\]

By Theorem \ref{main4},
\[
\begin{split}
&\frac{d}{dt}\int_{\varphi_t(M)}u_t\,d\mu_t-\frac{n-m}{2}a_{k_3}(T-t)\int_{\varphi_t(M)}u_t\,d\mu_t\\
&\leq\int_{\varphi_t(M)}u_t\left(-\frac{1}{2}\Delta_t^\perp U-\left|\frac{\nabla_t^\perp u_t}{u_t}-\meanvec_t\right|^2\right) d\mu_t.
\end{split}
\]

Since $\dot b_k=a_kb_k$, the result follows.
\end{proof}

\smallskip

\section{A generalization of the Aronzon-B\'enilan estimate}

The Aronzon-B\'enilan estimate \cite{ArBe} is a differential Harnack inequality for the porous medium equation
\[
\dot\rho_t=\Delta(\rho_t^m).
\]

In this section, we apply Theorem \ref{main1} and prove the following generalization of the Aronzon-B\'enilan estimate.

\begin{thm}\label{AB}
Assume that the Ricci curvature of the underlying compact Riemannian manifold $M$ is non-negative. Let $U$ be a function on $M$ satisfying
\[
\Delta U\geq \frac{k_3}{2m}
\]
where $m-1+\frac{2}{n}>0$. Then any smooth positive solution $\rho_t$ of the equation
\[
\dot\rho_t=\Delta(\rho_t^m)+U\rho_t^{2-m}.
\]
satisfies
\[
\frac{2m}{m-1}\Delta(\rho_t^{m-1})\leq \frac{2n}{2+n(m-1)}a_{\frac{k_3(2+n(m-1))}{2n}}(t).
\]
\end{thm}

\begin{proof}
A computation shows that $h_t=\frac{2m}{1-m}\rho_t^{m-1}$ satisfies
\[
\dot h_t+\frac{1}{2}|\nabla h_t|^2=\frac{1}{2}(1-m)h_t\Delta h_t-2mU.
\]
It follows that
\[
\begin{split}
&\Delta\left(\dot h_t+\frac{1}{2}|\nabla h_t|^2\right)+\frac{1}{2}(m-1)(\Delta h_t)^2+2m\Delta U\\
&=(1-m)\left<\nabla h_t,\nabla \Delta h_t\right>+m\rho_t^{m-1}\Delta\Delta h_t.
\end{split}
\]

The rest follows from the assumptions and Theorem \ref{main1}.
\end{proof}

\begin{rem}
Note that the above argument works regardless whether $m$ is greater than $1$ or not. 
\end{rem}

\smallskip

\section{On Laplacian and Hessian comparison type theorems}

In this section, we prove versions of Laplacian and Hessian type comparison theorems for the following cost function
\begin{equation}\label{newcost}
c_{s,t}(x,y)=\inf_{\gamma(s)=x,\gamma(t)=y}\int_s^t\frac{1}{2}|\dot\gamma(\tau)-\nabla U_1(\gamma(\tau))|^2-U_2(\gamma(\tau))d\tau.
\end{equation}
More precisely,

\begin{thm}\label{main7}
Assume that
\begin{enumerate}
  \item the Ricci curvature of the underlying manifold $M$ is non-negative,
  \item $\Delta\left(U_2-\frac{1}{2}|\nabla U_1|^2\right)\geq k_3$ for some negative constant $k_3$.
\end{enumerate}
Then the cost function $c_{0,t}$ defined by (\ref{cost}) satisfies
\[
\Delta_x c_{0,t}(x_0,x)\leq \sqrt{-k_3n}\coth\left(\sqrt{-\frac{k_3}{n}}\,t\right),
\]
wherever $c_{0,t}(x_0,\cdot)$ is twice differentiable.
\end{thm}

\begin{thm}\label{main8}
Assume that
\begin{enumerate}
  \item the sectional curvature of the underlying manifold $M$ is non-negative,
  \item $\nabla^2\left(U_2-\frac{1}{2}|\nabla U_1|^2\right)\geq k_3I$ for some negative constant $k_3$.
\end{enumerate}
Then the cost function $c_{0,t}$ defined by (\ref{cost}) satisfies
\[
\nabla_x^2 c_{0,t}(x_0,x)\leq \sqrt{-k_3}\coth\left(\sqrt{-k_3}\,t\right)I
\]
wherever $c_{0,t}(x_0,\cdot)$ is twice differentiable.
\end{thm}

\begin{rem}
  The function $x\mapsto c_{0,t}(x_0,x)$ is locally semi-concave. In particular, it is twice differentiable Lebesgue almost everywhere by Alexandrov's theorem. Therefore, the conclusions in Theorem \ref{main7} and \ref{main8} hold Lebesgue almost everywhere (see \cite{Vi1} for the definitions and the results).
\end{rem}

\begin{rem}
We can see that the above theorems are sharp by looking at the case $M=\Real^n$, $U_1\equiv 0$, and $U_2(x)=-\frac{k^2}{2}|x|^2$. We have
\[
\nabla^2\left(U_2-\frac{1}{2}|\nabla U_1|^2\right)= -k^2I,\quad \Delta\left(U_2-\frac{1}{2}|\nabla U_1|^2\right)= -k^2n
\]
which is the equality case in the second conditions of Theorem \ref{main7} and \ref{main8}. We also have
\[
\nabla^2_xc_{0,t}(0,x)=k\coth(kt)I,\quad \Delta_x c_{0,t}(0,x)=kn\coth(kt)
\]
which gives the equality case in the conclusions of Theorem \ref{main7} and \ref{main8}.
\end{rem}

\begin{rem}
A Bishop-Gromov type volume comparison theorem follows from Corollary \ref{maincor1}.
\end{rem}

The proof of Theorem \ref{main8} is similar to that of Theorem \ref{main7} and will be omitted.

\begin{proof}[Proof of Theorem \ref{main7}]
If $c_{0,t}(x_0,x)$ is smooth, then the result follows from Theorem \ref{main1} and \ref{main2}. Indeed, the Legendre transform of the Lagrangian
\[
L(x,v)=\frac{1}{2}g_{ij}(x)(v^i-g^{il}(x)(U_1)_{x_l}(x))(v^j-g^{jk}(x)U_{x_k}(x))-U_2(x)
\]
is given by
\[
\begin{split}
H(x,p)&=\sup_{v\in T_xM}[p(v)-L(x,v)]\\
&=\frac{1}{2}g^{ij}(x)p_ip_j+g^{ij}(x)p_i(U_1)_{x_j}(x)+U_2(x).
\end{split}
\]
Here we sum over repeated indices.

The corresponding Hamilton-Jacobi equation is given by
\begin{equation}\label{HJ}
\dot f_t+\frac{1}{2}|\nabla f_t|^2+\left<\nabla U_1,\nabla f_t\right> +U_2=0
\end{equation}
and $c_{0,t}(x_0,x)$ is a particular solution (see \cite{AgSa}).

If we set $X_t=\nabla\left(c_{0,t}(x_0,\cdot)+U_1\right)$, then
\[
\tr\left(\nabla\left( \dot X_t+\nabla_{X_t}X_t\right)+\nabla^2\left(U_2-\frac{1}{2}|\nabla U_1|^2\right)\right)=0.
\]

Therefore,
\[
\Delta_x c_{0,t}(x_0,x)\leq na_{\frac{k_3}{n}}(t)=\sqrt{-k_3n}\coth\left(\sqrt{-\frac{k_3}{n}}\,t\right)
\]
by Theorem \ref{main1}.

In general, if $x$ is a point where $c_{0,t}(x_0,\cdot)$ is twice differentiable, then there is a unique minimizer $\gamma$ to the infimum (\ref{cost}) joining $x_0$ and $x$. Moreover, $c_{0,s}(x_0,\cdot)$ is smooth at $\gamma(s)$ for each $s$ in $(0,t)$ (see \cite{CaSi}). Therefore, the proof of Theorem \ref{main1} still applies. Note that, in this case, (\ref{m5}) is an equality.
\end{proof}

\smallskip

\section{On the semigroup approach}

In this section, we give a semigroup proof of Theorem \ref{main3} and \ref{AB} which does not require any use of maximum principle. Such a proof was first given by \cite{BaLe}, assuming that the equation
\begin{equation}\label{eq}
\dot\rho_t=L\rho_t
\end{equation}
is given by an operator $L$ without constant term which is self-adjoint with respect to a weighted $L^2$ inner-product.

In the case of the heat equation, the key idea is to consider expressions of the form
\[
P_{T-t}(\rho_t|\nabla \rho_t|^2)\quad\text{ and }\quad P_{T-t}\dot\rho_t,
\]
where $P_t$ is the heat semigroup.

Since the heat semi-group is symmetric, it is equivalent to consider the followings instead
\begin{equation}\label{eq2}
\int_M\rho_t|\nabla \rho_t|^2 \varrho_{T-t} d\vol\quad\text{ and }\quad \int_M\dot\rho_t\varrho_{T-t} d\vol,
\end{equation}
where $\varrho_t$ ranges over all solutions of (\ref{eq}).

When $L$ is not self-adjoint but still linear, we also consider the expressions in (\ref{eq2}). However, in this case, $\varrho_t$ ranges over solutions of the equation
\[
\dot\varrho_t=L^*\varrho_t
\]
instead, where $L^*$ is the adjoint of $L$.

\begin{proof}[Proof of Theorem \ref{main3}]
Let $\varrho_t$ be a positive solution of the equation
\[
\dot\varrho_t=\Delta\varrho_t-\left<\nabla U_1,\nabla\varrho_t\right>+U_2\varrho_t.
\]

Let $k_t$ be a one-parameter family of smooth functions. A computation shows that
\[
\begin{split}
&\frac{d}{dt}\int_{M}\rho_t\,k_t\,\varrho_{T-t}\,d\vol=\int_{M}\dot\rho_t\,k_t\,\varrho_{T-t} +\rho_t\,\dot k_t\,\varrho_{T-t}-\rho_t\,k_t\,\dot\varrho_{T-t}\,d\vol\\
&=\int_{M}\rho_t\,(\dot k_t-\Delta k_t-2\left<\nabla f_t,\nabla k_t\right>)\,\varrho_{T-t}\,d\vol,
\end{split}
\]
where $f_t=\log\rho_t+\frac{1}{2} U_1$.

It follows that $c:=\int_{M}\rho_t\varrho_{T-t}\,d\vol$ is independent of $t$. By Bochner formula, we also have
\[
\begin{split}
&\frac{d}{dt}\int_{M}\rho_t(\Delta f_t)\,\varrho_{T-t}d\vol\\
&= \int_{M}\rho_t\left(2|\nabla^2 f_t|^2+2\ric(\nabla f_t,\nabla f_t)-\frac{1}{2}\Delta V\right)\varrho_{T-t}d\vol\\
&\geq \int_{M}\rho_t\left(\frac{2}{n}(\Delta f_t)^2-\frac{k^2n}{2}\right)\varrho_{T-t}d\vol\\
&\geq \int_{M}\rho_t\left(\frac{4a(t)\Delta f_t}{n}-\frac{2a(t)^2}{n}-\frac{k^2n}{2}\right)\varrho_{T-t}d\vol.
\end{split}
\]

So $b(t):=\int_{M}\rho_t(\Delta h_t)\,\varrho_{T-t}d\vol$ satisfies
\[
\dot b\geq \frac{4a(t)}{n}b(t)-\frac{2a(t)^2c}{n}-\frac{k^2nc}{2}.
\]

If $a(t)=-\frac{kn}{2}\coth(kt)$, then
\[
\dot b(t)\geq -2k b(t)\coth(kt)-\frac{k^2n(\coth^2(kt)+1)c}{2}.
\]
It follows that
\[
\begin{split}
\int_{M}\rho_t(\Delta f_t)\,\varrho_{T-t}d\vol&=b(t)\geq -\frac{kn}{2}\coth(kt)c\\
&=-\frac{kn}{2}\coth(kt)\int_{M}\rho_t\varrho_{T-t}\,d\vol.
\end{split}
\]

By setting $t=T$, we obtain
\[
\begin{split}
\int_M\rho_T(\Delta f_T)\,\varrho_0d\vol=-\frac{kn}{2}\coth(kt)\int_M\rho_T\varrho_0\,d\vol.
\end{split}
\]

Since $\varrho_0$ is arbitrary, we must have
\[
\dot f_t-|\nabla f_t|^2+\frac{1}{2}V=\Delta f_t\geq -\frac{kn}{2}\coth(kt).
\]
\end{proof}

\smallskip

\section{The generalized Li-Yau estimate without compactness assumption}\label{nocompact}

In this section, we give another prove of Theorem \ref{main} without making any compactness assumption. The proof uses the standard localization argument as in \cite{LiYa}.

\begin{proof}[Proof of Theorem \ref{main}]
Let $\rho_t$ be a positive solution of the equation $\dot \rho_t=\Delta \rho_t+\left<\nabla U_1\nabla \rho_t\right>+U_2 \rho_t$ and let $f_t=\log \rho_t+\frac{1}{2}U_1$. Then $\dot f_t=|\nabla f_t|^2 +U_2-\frac{1}{4}|\nabla U_1|^2-\frac{1}{2}\Delta U_1+\Delta f_t=|\nabla f_t|^2+\Delta f_t -\frac{1}{2}V$, where $V=-2U_2+\frac{1}{2}|\nabla U_1|^2+\Delta U_1$. It follows that
\[
\ddot f_t-\Delta \dot f_t-2\left<\nabla f_t,\nabla \dot f_t\right>=0
\]
and
\[
\begin{split}
&\frac{d}{dt} |\nabla f_t|^2-\Delta|\nabla f_t|^2-2\left<\nabla f_t,\nabla|\nabla f_t|^2\right>\\
&\leq -\left<\nabla f_t,\nabla V\right>-\frac{2}{n}\left(\dot f_t-|\nabla f_t|^2+\frac{1}{2}V\right)^2.
\end{split}
\]

Let $F_t=a_1\dot f_t+a_2|\nabla f_t|^2+a_3 V+a_4$, where
\[
a_2=\frac{(\exp(2tk)-1)^2}{\exp(2tk)}=(\exp(tk)-\exp(-tk))^2,
\]
$a_1=-\alpha a_2$, $a_3=-\frac{\alpha}{2}a_2$, $\alpha>1$. $a_4$ is a function of time $t$ to be determined.

Then
\[
\begin{split}
&\dot F_t-\Delta F_t-2\left<\nabla f_t,\nabla F_t\right>\\
&\leq -a_2\left<\nabla f_t,\nabla V\right>-\frac{2a_2}{n}\left(\dot f_t-|\nabla f_t|^2+\frac{1}{2}V\right)^2\\
&-a_3\Delta V-2a_3\left<\nabla f_t,\nabla V\right>-\alpha\dot a_2\dot f_t+\dot a_2|\nabla f_t|^2+\dot a_3 V+\dot a_4\\
&\leq -\frac{2a_2}{n}\left(\frac{1}{\alpha a_2}F_t+\left(1-\frac{1}{\alpha}\right)|\nabla f_t|^2-\left(\frac{a_3}{\alpha a_2}+\frac{1}{2}\right) V-\frac{a_4}{\alpha a_2}\right)^2\\
&-a_3\Delta V-(a_2+2a_3)\left<\nabla f_t,\nabla V\right>+\frac{\dot a_2}{a_2}F_t+\left(\dot a_3-\frac{\dot a_2a_3}{a_2}\right) V+\dot a_4-\frac{\dot a_2a_4}{a_2}.
\end{split}
\]

Let $G_t=\eta F_t$, where $\eta$ is a cut off function. Let us fix a time $t$. At a maximum point of $G_t$, we have
\[
\nabla F_t=-\frac{F_t}{\eta}\nabla \eta, \quad \Delta F_t\leq \frac{2F_t}{\eta^2}|\nabla \eta|^2-\frac{F_t}{\eta}\Delta\eta,\quad \dot F_t\geq 0.
\]

Therefore,
\[
\begin{split}
&-\frac{2F_t}{\eta^2}|\nabla \eta|^2+\frac{F_t}{\eta}\Delta\eta+2\frac{F_t}{\eta}\left<\nabla f_t,\nabla \eta\right>\\
&\leq -\frac{2a_2}{n}\left(\frac{1}{\alpha a_2}F_t+\left(1-\frac{1}{\alpha}\right)|\nabla f_t|^2-\left(\frac{a_3}{\alpha a_2}+\frac{1}{2}\right) V-\frac{a_4}{\alpha a_2}\right)^2\\
&-a_3\Delta V-(a_2+2a_3)\left<\nabla f_t,\nabla V\right>+\frac{\dot a_2}{a_2}F_t+\left(\dot a_3-\frac{\dot a_2a_3}{a_2}\right) V+\dot a_4-\frac{\dot a_2a_4}{a_2}.
\end{split}
\]

Let $r:[0,\infty)\to [0,1]$ be a function such that $r(x)=1$ if $x\leq R$, $r(x)=0$ if $x\geq 2R$, $r'\leq 0$, $\frac{r'(x)^2}{r(x)}\leq \frac{C}{R^2}$, and $|r''|\leq \frac{C}{R^2}$, where $C>0$ is a constant. Let us fix a point $x_0$ and let us denote the ball of radius $R$ centered at $x_0$ by $B_R$. Let $\eta=r(d(x_0,x))$, where $d(x_0,x)$ is the distance from $x_0$ to $x$. It follows that
\[
\frac{|\nabla\eta|^2}{\eta}\leq \frac{C}{R^2}.
\]

Since the Ricci curvature is non-negative, we have
\[
\Delta\eta\geq -\frac{C}{R^2}
\]
by the Laplacian comparison theorem.

Then
\[
\begin{split}
&-\frac{3CF_t}{\eta R^2}-\frac{2\sqrt C F_t^{3/2}}{\sqrt\eta R}\frac{|\nabla f_t|}{\sqrt{F_t}}\\
&\leq -\frac{2a_2}{n}\left(\frac{1}{\alpha a_2}F_t+\left(1-\frac{1}{\alpha}\right)|\nabla f_t|^2-\frac{a_4}{\alpha a_2}\right)^2\\
&+\frac{\alpha nk^2}{2}a_2+a_2(\alpha-1)|\nabla f_t||\nabla V|+\frac{\dot a_2}{a_2}F_t+\dot a_4-\frac{\dot a_2a_4}{a_2}.
\end{split}
\]

Let $H_t=\frac{|\nabla f_t|^2}{F_t}$ and assume that $|\nabla V|\leq C$, then
\[
\begin{split}
&-\frac{3CF_t}{\eta R^2}-\frac{2\sqrt C \sqrt{H_t} F_t^{3/2}}{\sqrt\eta R}\\
&\leq -\frac{2a_2F_t^2}{n}\left(\frac{1}{\alpha a_2}+\left(1-\frac{1}{\alpha}\right)H\right)^2+\frac{4a_4F_t}{n\alpha}\left(\frac{1}{\alpha a_2}+\left(1-\frac{1}{\alpha}\right)H_t\right)\\
&+a_2C(\alpha-1)\sqrt{F_t H_t}+\frac{\dot a_2}{a_2}F_t+\dot a_4-\frac{\dot a_2a_4}{a_2}-\frac{2a_4^2}{\alpha^2n a_2}+\frac{\alpha nk^2}{2}a_2.
\end{split}
\]

Let us choose $a_4$ such that $\frac{a_4}{a_2}=-\frac{n\alpha^{3/2}k}{2}\coth\left(\frac{kt}{\sqrt\alpha}\right)$. Note that $\frac{a_4}{a_2}$ satisfies the following Riccati equation
\[
\frac{d}{dt}\left(\frac{a_4}{a_2}\right)-\frac{2}{\alpha^2n}\left(\frac{a_4}{a_2}\right)^2+\frac{\alpha nk^2}{2}=0.
\]
Then
\[
\begin{split}
0&\leq -\frac{2a_2G_t^2}{n}\left(\frac{1}{\alpha a_2}+\left(1-\frac{1}{\alpha}\right)H_t\right)^2+\frac{4a_4\eta G_t}{n\alpha}\left(\frac{1}{\alpha a_2}+\left(1-\frac{1}{\alpha}\right)H_t\right)\\
&+a_2C(\alpha-1)\sqrt{G_t H_t}+\frac{\dot a_2}{a_2}\eta G_t+\frac{3CG_t}{ R^2}+\frac{2\sqrt C \sqrt{H_t} G_t^{3/2}}{R}.
\end{split}
\]

Since $a_4\leq 0$, it follows that
\[
\begin{split}
0&\leq -\frac{2a_2G_t^2}{n}\left(\frac{1}{\alpha a_2}+\left(1-\frac{1}{\alpha}\right)H_t\right)^2+\left(\frac{\dot a_2}{a_2}+\frac{4a_4}{n\alpha^2 a_2}\right)\eta G_t\\
&+a_2C(\alpha-1)\sqrt{G_t H_t}+\frac{3CG_t}{ R^2}+\frac{2\sqrt C \sqrt{H_t} G_t^{3/2}}{R}.
\end{split}
\]

Since $\alpha\geq 1$, a computation shows that
\[
\frac{\dot a_2}{a_2}+\frac{4a_4}{n\alpha^2 a_2}=2k\coth(kt)-\frac{2k}{\sqrt \alpha}\coth\left(\frac{kt}{\sqrt\alpha}\right)\geq 0.
\]

Therefore,
\[
\begin{split}
0&\leq -\frac{2a_2G_t^2}{n}\left(\frac{1}{\alpha a_2}+\left(1-\frac{1}{\alpha}\right)H_t\right)^2+\left(\frac{\dot a_2}{a_2}+\frac{4a_4}{n\alpha^2 a_2}\right) G_t\\
&+a_2C(\alpha-1)\sqrt{G_t H_t}+\frac{3CG_t}{ R^2}+\frac{2\sqrt C \sqrt{H_t} G_t^{3/2}}{R}.
\end{split}
\]

It follows that $G_t$ and hence the restriction of $F_t$ to the ball $B_R$ are less than or equal to the largest zero of the function
\[
\begin{split}
x\mapsto &-\frac{2a_2x^2}{n}+\left(\frac{\dot a_2}{a_2}+\frac{4a_4}{n\alpha^2 a_2}\right)B x\\
&+a_2AC(\alpha-1)\sqrt{x }+\frac{3BCx}{ R^2}+\frac{2A\sqrt C x^{3/2}}{R},
\end{split}
\]
where $A=\frac{\sqrt{H_t}}{\left(\frac{1}{\alpha a_2}+\left(1-\frac{1}{\alpha}\right)H_t\right)^2}$ and $B=\frac{1}{\left(\frac{1}{\alpha a_2}+\left(1-\frac{1}{\alpha}\right)H_t\right)^2}$

Since $B\leq \alpha^2a_2^2$ and $A\leq C$ are bounded independent of $R$, we can let $R\to\infty$. Therefore, $F_t$ is less than or equal to the largest zero of the function
\begin{equation}\label{eqn}
x\mapsto -\frac{2a_2x^2}{n}+\left(\alpha^2a_2\dot a_2+\frac{4a_2a_4}{n}\right) x+a_2AC(\alpha-1)\sqrt{x }.
\end{equation}

Now let $\alpha\to 1$ in (\ref{eqn}). Then we have $F_t\leq 0$. The result follows.
\end{proof}

\end{document}